\documentclass[preprint,sort&compress]{elsarticle}

\usepackage{amsfonts}
\usepackage{amsmath}
\newproof{proof}{Proof}
\newproof{pot}{Proof of Theorem \ref{thm2}}
\newtheorem{theorem}{Theorem}

\begin{document}

\title{Backlund transformations of curves in the Galilean and
pseudo-Galilean spaces}
\author[rvt]{S\"{u}leyman Cengiz\corref{cor1}\fnref{fn1}}
\ead{suleymancengiz@karatekin.edu.tr}
\author[focal]{Nevin G\"{u}rb\"{u}z\fnref{fn2}}
\ead{toprak400@gmail.com}
\address[rvt]{Karatekin University, Mathematics Department, \c{C}ank{\i}r{\i} }
\address[focal]{Eski\c{s}ehir Osmangazi University, Mathematics Department, Eski\c{s}ehir}
\begin{abstract}
Backlund transformations of admissible curves in the
Galilean 3-space and pseudo-Galilean 3-space and also spatial
Backlund transformations of space curves in Galilean 4-space preserve the torsions under certain assumptions.
\end{abstract}

\maketitle

\section{\textbf{Introduction}}

In the 1890s Bianchi, Lie, and finally Backlund looked at what are now
called Backlund transformations of surfaces. In modern parlance, they begin
with two surfaces in Euclidean space in a line congruence: there is a
mapping between the surfaces M$_{1}$ and M$_{2}$ such that the line through
any two corresponding points is tangent to both surfaces. Backlund proved
that if a line congruence satisfied two additional conditions, that the line
segment joining corresponding points has constant length, and that the
normals at corresponding points form a constant angle, then the two surfaces
are necessarily surfaces of constant negative curvature. He was also able to
show that a Backlund transformation is integrable, in the sense that given a
point on a surface of constant negative curvature and a tangent line segment
at that point, a new surface of constant negative curvature can be found,
containing the endpoint of the line segment, that is a Backlund transform of
the original surface.

The classical Backlund theorem studies the transformation of surfaces of
constant negative curvature in $\mathbb{R}^{3}$ by realizing them as the
focal surfaces of a pseudo-spherical line congruence. The integrability
theorem says that we can construct a new surface in $\mathbb{R}^{3}$ with
constant negative curvature from a given one. In \cite{tenenblat1} Tenenblat and Terng
established a high dimension generalization of Backlund's theorem which is
very interesting both for physical and mathematical reasons. 
After that Chern and Terng customized Backlund theorem for affine surfaces 
\cite{chern}. By the same year this transformation was reduced to
corresponding asymptotical lines by Terng \cite{terng} and following years
Tenenblat expanded the Backlund transformation of two surfaces in $\mathbb{R}
_{1}^{3}$ to space forms\cite{tenenblat2}. In 1990 Palmer constructed a
Backlund transformation between spacelike and timelike surfaces of constant
negative curvature in $\mathbb{E}_{1}^{3}$ \cite{palmer}. At that decade
some researchers gave Backlund transformations on Weingarten surfaces \cite
{buyske, tian, chen, cao}.

In 1998 Calini and Ivey \cite{calini} proposed a geometric realization of the Backlund
Transformation for the sine-Gordon equation in the context of curves of
constant torsion. Since the asymptotic lines on a pseudospherical surface
have constant torsion, the Backlund transformation can be restricted to get
a transformation that carries constant torsion curves to constant torsion
curves. Later the converse of the idea was proved and
generalized for the n-dimensional case by Nemeth \cite{nemeth1}. In \cite{nemeth2}
Nemeth studied a similar concept for constant torsion curves in the
3-dimensional constant curvature spaces. Shief and Rogers
used an analogue of the classical Backlund transformation for the generation
of soliton surfaces \cite{shief}. In \cite{chou} Chou, Kouhua and Yongbo
obtained the Backlund transformation on timelike surfaces with constant mean
curvature in $\mathbb{R}_{1}^{2}.$ Zuo, Chen, Cheng studied Backlund
theorems in three dimensional de Sitter space and anti-de Sitter space \cite
{zuo}. Abdel-Baky presented the Minkowski versions of the Backlund theorem
and its application by using the method of moving frames \cite{abdel}. G\"{u}rb\"{u}z 
studied Backlund transformations in $\mathbb{R}_{1}^{n}$ \cite
{gurbuz}. Using the same method \"{O}zdemir and \c{C}\"{o}ken have
studied Backlund transformations of non-lightlike constant torsion curves in
Minkowski 3-space\cite{ozdemir}.

In this paper we show that a restriction of Backlund theorem on space curves satisfying the given three conditions preserves the torsions of the curves in Galilean and pseudo-Galilean spaces. For the necessary definitions and theorems of Galilean and pseudo-Galilean spaces we refered \cite{yaglom, suha, divjak, pavkovic}

\section{\textbf{Preliminaries}}

The Galilean space $\mathbb{G}^{3}$ is the three dimensional real affine
space with the absolute figure $\{w,f,I\}$ , where $w$ is the ideal plane, $
\ f$ is a line in $w$ and $I$ is the fixed elliptic involution of points of $
f.$

The scalar product of two vectors $X=(a_{1},a_{2},a_{3})$ and $
Y=(b_{1},b_{2},b_{3})$ in $\mathbb{G}^{3}$ is defined by
\begin{center}
$<X,Y>_{G}=\left\{ 
\begin{array}{ll}
a_{1}.b_{1} & a_{1}\neq 0\text{ or }b_{1}\neq 0 \\ 
a_{2}\text{\textperiodcentered }b_{2}+a_{3}\text{\textperiodcentered }b_{3},
& a_{1}=0\text{ and }b_{1}=0
\end{array}
\right. $\\[0pt]
\end{center}

An admissible curve $\alpha :I\subset R\rightarrow G^{3}$ of the class $C^{\
r}$ (r $\geq $ 3) in the Galilean space $G^{3}$ is defined by the
parametrization\ 
\begin{equation*}
\alpha (s)=(s,x(s),y(s))
\end{equation*}
where s is the arc length of $\alpha $ with the differential form $ds=dx$.
The curvature $\kappa (s)$ and the torsion $\mathbb{\tau }(s)$ of an
admissible curve in $G^{3}$ are given by $\kappa (s)=\sqrt{y^{\prime \prime
2}(s)-z^{\prime \prime 2}(s)}\ $and $\mathbb{\tau }(s)=(\det (\alpha
^{\prime }(s),\alpha ^{\prime \prime }(s),\alpha ^{\prime \prime \prime
}(s))/\kappa ^{2}(s)$ respectively. The associated moving trihedron is given
by
\begin{eqnarray*}
E_{1} &=&\alpha ^{\prime }(s)=(1,x^{\prime }(s),y^{\prime }(s)) \\
E_{2} &=&\frac{(0,x^{\prime \prime }(s),y^{\prime \prime }(s))}{\sqrt{
x^{\prime \prime 2}(s)+y^{\prime \prime 2}(s)}} \\
E_{3} &=&\frac{(0,-y^{\prime \prime }(s),x^{\prime \prime }(s))}{\sqrt{
x^{\prime \prime 2}(s)+y^{\prime \prime 2}(s)}}
\end{eqnarray*}
Then the  Frenet formulas in the Galilean space $G^{3}$ becomes:
\begin{eqnarray}
\mathbf{E}_{1}^{\prime } &=&\kappa \mathbf{E}_{2}  \notag \\
\mathbf{E}_{2}^{\prime } &=&\tau \mathbf{E}_{3}   \label{frenet1} \\
\mathbf{E}_{3}^{\prime } &=&-\tau \mathbf{E}_{2}  \notag
\end{eqnarray}
\section{\textbf{Backlund transformations of admissible curves in the
Galilean space G}$^{3}$}
\begin{theorem}
Suppose that $\psi $ is a transformation between two admissible curves $
\alpha $ and $\widetilde{\alpha }$ in the Galilean space $\mathbb{G}^{3}$
with $\widetilde{\alpha }=\psi (\alpha )$ such that in the corresponding
points we have:
\begin{enumerate}[i.]
\item The line segment $[\widetilde{\alpha }(s)\alpha (s)]$ at the
intersection of the osculating planes of the curves has constant length $r$
\item The distance vector $\widetilde{\alpha }(s)-\alpha (s)$ has the same
angle $\gamma \neq \frac{\pi }{2}$ with the tangent vectors of the curves
\item The binormals of the curves have the same constant angle $\phi
\neq 0$.
\end{enumerate}
Then these curves are congruent with the curvatures and torsions
\begin{eqnarray*}
\tilde{\kappa} &=&\kappa =-2\frac{d\gamma }{ds} \\
\widetilde{\tau } &=&\tau =\frac{\sin \phi }{r}
\end{eqnarray*}
and the transformation of the curves is given by
\begin{equation*}
\widetilde{\alpha }=\alpha +\frac{2C}{\tau ^{2}+C^{2}}(\cos \gamma \mathbf{E}
_{1}+\sin \gamma \mathbf{E}_{2})
\end{equation*}
where $C=\tau \tan \left( \frac{\phi }{2}\right) $ is a constant and $\gamma 
$ is a solution of the differential equation
\begin{equation*}
\frac{d\gamma }{ds}=\tau \sin \gamma \tan \frac{\phi }{2}
\end{equation*}
\end{theorem}
\begin{proof}
Denote by $(\mathbf{E}_{1},\mathbf{E}_{2},\mathbf{E}_{3})$ and $(\widetilde{
\mathbf{E}_{1}},\widetilde{\mathbf{E}_{2}},\widetilde{\mathbf{E}_{3}})$\ the
Frenet frames of the curves $\alpha$ and $\widetilde{\alpha }$ in the Galilean space $\mathbb{G}^{3}.$ Let $\widetilde{\mathbf{E}_{3}}$ be a unit binormal of $\widetilde{\alpha }$.

If we denote by $W_{1}$ the unit vector of $\widetilde{\alpha }-\alpha $,
then we can complete $W_{1},\mathbf{E}_{3}$ and $W_{1},\widetilde{\mathbf{E}
_{3}}$ to the positively oriented orthonormal frames $(W_{1},W_{2},W_{3})$
and $(W_{1},\widetilde{W}_{2},\widetilde{W_{3}})$ where $W_{3}=\mathbf{E}_{3},
\widetilde{W_{3}}=\widetilde{\mathbf{E}_{3}}$ and $\gamma $ is the angle
between $W_{1}$ and $\mathbf{E}_{1}.$ The frames $(W_{1},W_{2},W_{3})$ and $
(W_{1},\widetilde{W}_{2},\widetilde{W_{3}})$ can be obtained by rotating the
frames $(\mathbf{E}_{1},\mathbf{E}_{2},\mathbf{E}_{3})$ and $(\widetilde{
\mathbf{E}_{1}},\widetilde{\mathbf{E}_{2}},\widetilde{\mathbf{E}_{3}})$
around $\mathbf{E}_{3}$ and $\widetilde{\mathbf{E}_{3}}$ with an angle $\gamma$ respectively. So we
can write
\begin{equation*}
\left[ 
\begin{array}{l}
W_{1} \\ 
W_{2} \\ 
W_{3}
\end{array}
\right] =\left[ 
\begin{array}{lll}
\cos \gamma  & \sin \gamma  & 0 \\ 
-\sin \gamma  & \cos \gamma  & 0 \\ 
0 & 0 & 1
\end{array}
\right] \left[ 
\begin{array}{l}
\mathbf{E}_{1} \\ 
\mathbf{E}_{2} \\
\mathbf{E}_{3}
\end{array}
\right] 
\end{equation*}
and
\begin{equation*}
\left[ 
\begin{array}{l}
W_{1} \\ 
\widetilde{W_{2}} \\ 
\widetilde{W_{3}}
\end{array}
\right] =\left[ 
\begin{array}{lll}
\cos \gamma  & \sin \gamma  & 0 \\ 
-\sin \gamma  & \cos \gamma  & 0 \\ 
0 & 0 & 1
\end{array}
\right] \left[ 
\begin{array}{l}
\widetilde{\mathbf{E}_{1}} \\ 
\widetilde{\mathbf{E}_{2}} \\ 
\widetilde{\mathbf{E}_{3}}
\end{array}
\right] .
\end{equation*}
Similarly for a rotation around $W_{1}$ by the angle $\phi$
\begin{eqnarray*}
\widetilde{W_{2}} &=&\cos \phi W_{2}-\sin \phi W_{3} \\
\widetilde{W_{3}} &=&\sin \phi W_{2}+\cos \phi W_{3}
\end{eqnarray*}
From the above equations we write
\begin{eqnarray}
\widetilde{\mathbf{E}_{1}}&=&(\cos ^{2}\gamma +\sin ^{2}\gamma \cos \phi )
\mathbf{E}_{1}+\cos \gamma \sin \gamma (1-\cos \phi )\mathbf{E}_{2}\notag \\
&&+\sin \gamma \sin \phi \mathbf{E}_{3}\notag \\ 
\widetilde{\mathbf{E}_{2}}&=&\cos \gamma \sin \gamma (1-\cos \phi )\mathbf{E}
_{1}+(\sin ^{2}\gamma +\cos ^{2}\gamma \cosh \phi )\mathbf{E}_{2}\label{trihedron1}\\
&&-\cos \gamma \sin \phi \mathbf{E}_{3}\notag \\ 
\widetilde{\mathbf{E}_{3}}&=&-\sin \gamma \sin \phi \mathbf{E}_{1}
+\sin \phi \cos \gamma \mathbf{E}_{2}+\cos \phi \mathbf{E}_{3}\notag
\end{eqnarray}
Using (\ref{frenet1}) and (\ref{trihedron1}) for $\widetilde{\mathbf{E}_{3}}$
\begin{eqnarray*}
\frac{d\widetilde{\mathbf{E}_{3}}}{ds} &=&-\widetilde{\tau }\widetilde{
\mathbf{E}_{2}} \\
&=&(-\widetilde{\tau }\cos \gamma \sin \gamma (1-\cos \phi ))\mathbf{E}_{1}
\\
&&+(-\widetilde{\tau }(\sin ^{2}\gamma +\cos ^{2}\gamma \cos \phi ))\mathbf{E
}_{2} \\
&&+(\widetilde{\tau }\sinh \phi \cos \gamma )\mathbf{E}_{3}
\end{eqnarray*}
and taking derivative of $\widetilde{\mathbf{E}_{3}}$ in  (\ref{trihedron1}) with respect to s 
\begin{eqnarray}
\frac{d\widetilde{\mathbf{E}_{3}}}{ds} &=&(-\sin \phi \cos \gamma \frac{
d\gamma }{ds})\mathbf{E}_{1}  \notag \\
&&+(-\tau \cos \phi -\sin \gamma \sin \phi (\kappa +\frac{d\gamma }{ds}))
\mathbf{E}_{2}  \notag \\
&&+(\tau \sin \phi \cos \gamma )\mathbf{E}_{3}  \notag
\end{eqnarray}
then equating the two statements above we obtain
\begin{equation*}
\widetilde{\tau }=\tau 
\end{equation*}
\begin{equation*}
\frac{d\gamma }{ds}=\tau \sin \gamma \tanh \frac{\phi }{2}
\end{equation*}
Similarly, differentiating $\widetilde{\mathbf{E}_{1}}$ and $\widetilde{
\mathbf{E}_{2}}$ from (\ref{trihedron1}) and using  (\ref{frenet1})
\begin{equation*}
\tilde{\kappa}=\kappa =-2\frac{d\gamma }{ds}
\end{equation*}
Now $\alpha $ is a unit speed curve. Differentiating
\begin{equation*}
r^{2}=\left( \tilde{\alpha}-\alpha \right) ^{2}
\end{equation*}
and substituting the distance vector
\begin{equation}
\widetilde{\alpha }-\alpha =r(\cos \gamma \mathbf{E}_{1}+\sin \gamma \mathbf{
E}_{2})  \label{curve1}
\end{equation}
we find that $\tilde{\alpha}$ is also a unit speed curve.\\
Next taking the derivative of (\ref{curve1}) we obtain:
\begin{equation}
\widetilde{\mathbf{E}_{1}}=(1-r\sin \gamma \frac{d\gamma }{ds})\mathbf{E}
_{1}+r\cos \gamma (\kappa +\frac{d\gamma }{ds})\mathbf{E}_{2}+\tau r\sin
\gamma \mathbf{E}_{3}  \notag
\end{equation}
From this equation and the Frenet frames (\ref{trihedron1}) 
\begin{equation*}
\widetilde{\tau }=\tau =\frac{\sin \phi }{r}
\end{equation*}
Then rearranging this equality we get
\begin{equation*}
r=\frac{2\tau \tan \left( \frac{\phi }{2}\right) }{\tau ^{2}\left( 1+\tan
^{2}\left( \frac{\phi }{2}\right) \right) }
\end{equation*}
Finally with the aid of (\ref{curve1}) , naming the constant $C=\tau \tan
\left( \frac{\phi }{2}\right) $,\ the Backlund transformation of the curves is
\begin{equation*}
\widetilde{\alpha }=\alpha +\frac{2C}{\tau ^{2}+C^{2}}(\cos \gamma \mathbf{E}
_{1}+\sin \gamma \mathbf{E}_{2}).
\end{equation*}
\end{proof}

\section{\textbf{Backlund Transformations of admissible curves in the
pseudo-Galilean space }$\mathbf{G}_{1}^{3}$}
The pseudo-Galilean space $\mathbb{G}_{1}^{3}$ is the three dimensional real
affine space with the absolute figure $\{w,f,I\}$ , where $w$ is the ideal
plane, $\ f$ is a line in $w$ and $I$ is the fixed hyperbolic involution of
the points of $f.$

The scalar product of two vectors $X=(a_{1},a_{2},a_{3})$ and $
Y=(b_{1},b_{2},b_{3})$ in $\mathbb{G}_{1}^{3}$ is defined by
\begin{center}
$<X,Y>_{G}=\left\{ 
\begin{array}{ll}
a_{1}.b_{1} & a_{1}\neq 0\text{ or }b_{1}\neq 0 \\ 
a_{2}\text{\textperiodcentered }b_{2}-a_{3}\text{\textperiodcentered }b_{3},
& a_{1}=0\text{ and }b_{1}=0
\end{array}
\right. $\\
\end{center}
The curvature $\kappa (s)$ and the torsion $\mathbb{\tau }(s)$ of an
admissible curve $\alpha (s)=(s,x(s),y(s))$ in $\mathbb{G}_{1}^{3}$ are
given by
\begin{equation*}
\kappa (s)=\sqrt{\left\vert x^{\prime \prime 2}(s)-y^{\prime \prime
2}(s)\right\vert }\ \text{and }\mathbb{\tau }(s)=(\det (\alpha ^{\prime
}(s),\alpha ^{\prime \prime }(s),\alpha ^{\prime \prime \prime }(s))/\kappa
^{2}(s)\newline
\end{equation*}
respectively. The associated moving trihedron is given by 
\begin{eqnarray}
E_{1} &=&a^{\prime }(s)=(1,x^{\prime }(s),y^{\prime }(s))  \notag \\
E_{2} &=&\frac{(0,x^{\prime \prime }(s),y^{\prime \prime }(s))}{\sqrt{
\left\vert x^{\prime \prime 2}(s)-y^{\prime \prime 2}(s)\right\vert }} \notag \\
E_{3} &=&\frac{(0,\varepsilon y^{\prime \prime }(s),\varepsilon z^{\prime
\prime }(s))}{\sqrt{\left\vert x^{\prime \prime 2}(s)-y^{\prime \prime
2}(s)\right\vert }} \notag
\end{eqnarray}
where $\ \varepsilon = \mp $ 1.
The Frenet formulas in the pseudo-Galilean space $\mathbb{G}_{1}^{3}$ have
the following form:
\begin{eqnarray}
\mathbf{E}_{1}^{\prime } &=&\kappa \mathbf{E}_{2}  \notag \\
\mathbf{E}_{2}^{\prime } &=&\tau \mathbf{E}_{3}  \label{1.1} \\
\mathbf{E}_{3}^{\prime } &=&\tau \mathbf{E}_{2}  \notag
\end{eqnarray}

\subsection{\textbf{Backlund transformations of admissible curves which have
timelike binormals in the pseudo-Galilean space }$\mathbb{G}_{1}^{3}$\textbf{
:}}

\begin{theorem}
Suppose that $\psi $ is a transformation between two admissible curves $
\alpha $ and $\widetilde{\alpha }$ in the pseudo-Galilean space $\mathbb{G}
_{1}^{3}$ with $\widetilde{\alpha }=\psi (\alpha )$ such that in the
corresponding points we have:

\textbf{i.} The line segment $[\widetilde{\alpha }(s)\alpha (s)]$ at the
intersection of the osculating planes of the curves has constant length $r$

\textbf{ii. }The distance vector $\widetilde{\alpha }-\alpha $ has the same angle $
\gamma \neq \frac{\pi }{2}$ with the tangent vectors of the curves

\textbf{iii}. The timelike binormals of the curves have the same constant
angle $\phi \neq 0$. \newline
Then these curves have equal torsions
\begin{equation*}
\widetilde{\tau }=\tau =-\frac{\sinh \phi }{r}
\end{equation*}
and the Backlund transformation of the curves is
\begin{equation*}
\widetilde{\alpha }=\alpha +\frac{2C}{C^{2}-\tau ^{2}}(\cos \gamma \mathbf{E}
_{1}+\sin \gamma \mathbf{E}_{2})
\end{equation*}
where $C=\tau \tanh \left( \frac{\phi }{2}\right) $ is a constant and $\gamma $ is a solution of the differential equation
\begin{equation*}
\frac{d\gamma }{ds}=\tau\sin \gamma \tanh \frac{\phi }{2}.
\end{equation*}
\end{theorem}

\begin{proof}
Denote by $(\mathbf{E}_{1},\mathbf{E}_{2},\mathbf{E}_{3})$ and $(\widetilde{
\mathbf{E}_{1}},\widetilde{\mathbf{E}_{2}},\widetilde{\mathbf{E}_{3}})$\ the
Frenet frames of the curves $\alpha$ and $\widetilde{\alpha }$ in the pseudo-Galilean space 
$\mathbb{G}_{1}^{3}$ respectively. Let $\widetilde{\mathbf{E}_{3}}$ be a unit timelike
binormal of $\widetilde{\alpha }$ such that $\left\langle \widetilde{\mathbf{
E}_{3}},\widetilde{\mathbf{E}_{3}}\right\rangle =-1.$ For the rotations of frames we can write
\begin{equation*}
\left[ 
\begin{array}{l}
W_{1} \\ 
W_{2} \\ 
W_{3}
\end{array}
\right] =\left[ 
\begin{array}{lll}
\cos \gamma & \sin \gamma & 0 \\ 
-\sin \gamma & \cos \gamma & 0 \\ 
0 & 0 & 1
\end{array}
\right] \left[ 
\begin{array}{l}
\mathbf{E}_{1} \\ 
\mathbf{E}_{2} \\ 
\mathbf{E}_{3}
\end{array}
\right],
\end{equation*}
\begin{equation*}
\left[ 
\begin{array}{l}
W_{1} \\ 
\widetilde{W_{2}} \\ 
\widetilde{W_{3}}
\end{array}
\right] =\left[ 
\begin{array}{lll}
\cos \gamma & \sin \gamma & 0 \\ 
-\sin \gamma & \cos \gamma & 0 \\ 
0 & 0 & 1
\end{array}
\right] \left[ 
\begin{array}{l}
\widetilde{\mathbf{E}_{1}} \\ 
\widetilde{\mathbf{E}_{2}} \\ 
\widetilde{\mathbf{E}_{3}}
\end{array}
\right]{}
\end{equation*}
and
\begin{eqnarray*}
\widetilde{W_{2}} &=&\cosh \phi W_{2}+\sinh \phi W_{3} \\
\widetilde{W_{3}} &=&\sinh \phi W_{2}+\cosh \phi W_{3}
\end{eqnarray*}
From the equations above we can write
\begin{eqnarray}
\widetilde{\mathbf{E}_{1}} &=&(\cos ^{2}\gamma +\sin ^{2}\gamma \cosh \phi )
\mathbf{E}_{1}+\cos \gamma \sin \gamma (1-\cosh \phi )\mathbf{E}_{2}  \notag \\
&&-\sin \gamma \sinh \phi \mathbf{E}_{3}  \notag \\
\widetilde{\mathbf{E}_{2}} &=&\cos \gamma \sin \gamma (1-\cosh \phi )\mathbf{
E}_{1}+(\sin ^{2}\gamma +\cos ^{2}\gamma \cosh \phi )\mathbf{E}_{2}\label{curve2}\\
&&+\cos \gamma \sinh \phi \mathbf{E}_{3}  \notag \\
\widetilde{\mathbf{E}_{3}} &=&-\sin \gamma \sinh \phi \mathbf{E}_{1}+\sinh
\phi \cos \gamma \mathbf{E}_{2}+\cosh \phi \mathbf{E}_{3}  \notag
\end{eqnarray}
Differentiating $\widetilde{\mathbf{E}_{3}}$ with respect to the arc length
s and using the Frenet equations (\ref{1.1}) for $\widetilde{\mathbf{E}_{3}}$ we find
\begin{equation*}
\widetilde{\tau }=\tau
\end{equation*}
\begin{equation*}
\frac{d\gamma }{ds}=\tau \sin \gamma \tanh \frac{\phi }{2}
\end{equation*}
Next taking the derivative of the distance vector 
\begin{equation*}
\widetilde{\alpha }-\alpha =r(\cos \gamma \mathbf{E}_{1}+\sin \gamma \mathbf{
E}_{2})
\end{equation*}
and by (\ref{curve2}) we get
\begin{equation*}
\widetilde{\tau }=\tau =-\frac{\sinh \phi }{r}
\end{equation*}
Then rearranging the equality above
\begin{equation*}
r=\frac{2\tau \tanh \left( \frac{\phi }{2}\right) }{\tau ^{2}\left( \tanh
^{2}\left( \frac{\phi }{2}\right) -1\right) }
\end{equation*}
Finally with the aid of distance vector, naming the constant $C=\tau \tanh
\left( \frac{\phi }{2}\right) $,\ the Backlund transformation is obtained as
\begin{equation*}
\widetilde{\alpha }=\alpha +\frac{2C}{C^{2}-\tau ^{2}}(\cos \gamma \mathbf{E}
_{1}+\sin \gamma \mathbf{E}_{2})
\end{equation*}
\end{proof}

\subsection{\textbf{Backlund transformations of admissible curves which have
timelike normals in the pseudo-Galilean space }$\mathbb{G}_{1}^{3}$\textbf{:}}

\begin{theorem}
Suppose that $\psi $ is a transformation between two admissible curves $
\alpha $ and $\widetilde{\alpha }$ in the pseudo-Galilean space $\mathbb{G}
_{1}^{3}$ with $\widetilde{\alpha }=\psi (\alpha )$ such that in the
corresponding points we have:

\textbf{i.} The line segment $[\widetilde{\alpha }(s)\alpha (s)]$ at the
intersection of the osculating planes of the curves has constant length $r$

\textbf{ii. }The distance vector $\widetilde{\alpha }-\alpha $ has the same angle $
\gamma \neq 0$ with the tangent vectors of the curves

\textbf{iii}. The timelike normals of the curves have the same constant
angle $\phi \neq 0$. \newline
Then these curves have the relation between their torsions
\begin{equation*}
\widetilde{\tau }=-\tau =-\frac{\sinh \phi }{r}
\end{equation*}
and the Backlund transformation of the curves is given by
\begin{equation*}
\widetilde{\alpha }=\alpha +\frac{2C}{C^{2}-\tau ^{2}}(\cosh \gamma \mathbf{E
}_{1}+\sinh \gamma \mathbf{E}_{2})
\end{equation*}
where $C=\tau \tanh \left( \frac{\phi }{2}\right) $ is a constant and $
\gamma $ is a solution of the differential equation
\begin{equation*}
\frac{d\gamma }{ds}=-\tau \sinh \gamma \tanh \frac{\phi }{2}
\end{equation*}
\end{theorem}
\begin{proof}
Denote by $(\mathbf{E}_{1},\mathbf{E}_{2},\mathbf{E}_{3})$ and $(\widetilde{
\mathbf{E}_{1}},\widetilde{\mathbf{E}_{2}},\widetilde{\mathbf{E}_{3}})$\ the
Frenet frames of the curves $\alpha$ and $\widetilde{\alpha }$ in the pseudo-Galilean space 
$\mathbb{G}_{1}^{3}$ respectively.
Let $\widetilde{\mathbf{E}_{2}}$ be a unit timelike normal of $\widetilde{\alpha }
$ such that $\left\langle \widetilde{\mathbf{E}_{2}},\widetilde{\mathbf{E}
_{2}}\right\rangle =-1.$\newline
Again by the notation of previous proof it can be writen
\begin{equation*}
\left[ 
\begin{array}{l}
W_{1} \\ 
W_{2} \\ 
W_{3}
\end{array}
\right] =\left[ 
\begin{array}{lll}
\cosh \gamma & \sinh \gamma & 0 \\ 
\sinh \gamma & \cosh \gamma & 0 \\ 
0 & 0 & 1
\end{array}
\right] \left[ 
\begin{array}{l}
\mathbf{E}_{1} \\ 
\mathbf{E}_{2} \\ 
\mathbf{E}_{3}
\end{array}
\right],
\end{equation*}
\begin{equation*}
\left[ 
\begin{array}{l}
W_{1} \\ 
\widetilde{W_{2}} \\ 
\widetilde{W_{3}}
\end{array}
\right] =\left[ 
\begin{array}{lll}
\cosh \gamma & \sinh \gamma & 0 \\ 
\sinh \gamma & \cosh \gamma & 0 \\ 
0 & 0 & 1
\end{array}
\right] \left[ 
\begin{array}{l}
\widetilde{\mathbf{E}_{1}} \\ 
\widetilde{\mathbf{E}_{2}} \\ 
\widetilde{\mathbf{E}_{3}}
\end{array}
\right]
\end{equation*}
and
\begin{eqnarray*}
\widetilde{W_{2}} &=&\cosh \phi W_{2}+\sinh \phi W_{3} \\
\widetilde{W_{3}} &=&-\sinh \phi W_{2}+\cosh \phi W_{3}
\end{eqnarray*}
From the above equations we write
\begin{eqnarray}\label{frenet3}
\widetilde{\mathbf{E}_{1}} &=&(\cosh ^{2}\gamma -\sinh ^{2}\gamma \cosh \phi
)\mathbf{E}_{1}+\cosh \gamma \sinh \gamma (1-\cosh \phi )\mathbf{E}_{2}\notag \\
&&-\sinh \gamma \sinh \phi \mathbf{E}_{3}  \notag \\
\widetilde{\mathbf{E}_{2}} &=&\cosh \gamma \sinh \gamma (-1+\cosh \phi )
\mathbf{E}_{1}+(-\sinh ^{2}\gamma +\cosh ^{2}\gamma \cosh \phi )\mathbf{E}_{2} \\
&&+\cosh \gamma \sinh \phi \mathbf{E}_{3}  \notag \\
\widetilde{\mathbf{E}_{3}} &=&-\sinh \gamma \sinh \phi \mathbf{E}_{1}-\sinh
\phi \cosh \gamma \mathbf{E}_{2}+\cosh \phi \mathbf{E}_{3}  \notag
\end{eqnarray}
Differentiating $\widetilde{\mathbf{E}_{3}}$ with respect to the arc length
s and using Frenet equation for $\widetilde{\mathbf{E}_{3}}$ we find
\begin{equation*}
\widetilde{\tau }=-\tau
\end{equation*}
\begin{equation*}
\frac{d\gamma }{ds}=-\tau\sinh \gamma \tanh \frac{\phi }{2}
\end{equation*}
Next taking the derivative of the distance vector 
\begin{equation*}
\widetilde{\alpha }-\alpha =r(\cosh \gamma \mathbf{E}_{1}+\sinh \gamma 
\mathbf{E}_{2})
\end{equation*}
and from (\ref{frenet3}) it can be found 
\begin{equation}
\widetilde{\tau }=-\tau =\frac{\sinh \phi }{r}\notag
\end{equation}
Then rearranging the equality above we get
\begin{equation*}
r=\frac{2\tau \tanh \left( \frac{\phi }{2}\right) }{\tau ^{2}\left( \tanh
^{2}\left( \frac{\phi }{2}\right) -1\right) }
\end{equation*}
Finally with the aid of distance vector, naming the constant $C=\tau \tanh
\left( \frac{\phi }{2}\right) $,\ the transformation is obtained as
\begin{equation*}
\widetilde{\alpha }=\alpha +\frac{2C}{C^{2}-\tau ^{2}}(\cosh \gamma \mathbf{E
}_{1}+\sinh \gamma \mathbf{E}_{2})
\end{equation*}
\end{proof}
\section{\textbf{Spatial Backlund transformations of curves in Galilean
space G}$^{4}$}
The Galilean space $\mathbb{G}^{4}$ consists of a four dimensional real affine space endowed with global absolute time and Euclidean metric structure E over the simultaneity hyperplanes defined as the three-dimensional real affine spaces with underlying vector space Ker(t) of the absolute time functional which is a non zero linear functional t : V $\to$ R on the underlying vector
space V of E.

The scalar product of two vectors $X=(a_{1},a_{2},a_{3},a_{4})$ and $
Y=(b_{1},b_{2},b_{3},b_{4})$ in $\mathbb{G}^{4}$ is defined by
\begin{center}
$<X,Y>_{G}=\left\{ 
\begin{array}{ll}
a_{4}.b_{4} & a_{4}\neq 0\text{ or }b_{4}\neq 0 \\
a_{1}\text{\textperiodcentered }b_{1}+a_{2}\text{\textperiodcentered }
b_{2}+a_{3}\text{\textperiodcentered }b_{3}, & a_{4}=0\text{ and }b_{4}=0
\end{array}
\right. $\\
\end{center}
Let $\alpha (s)=(x(s),y(s),z(s),t(s))$ be the position vector
of a curve. Then the condition $\alpha ^{\prime }=E_{1},|E_{1}|=1$ is
equivalent to the condition $t\left( s\right) =s.$ Thus natural equations
of a curve $\alpha (s)=(x(s),y(s),z(s),s)$ in ${G}^{4}$ are
\begin{center}
$\kappa (s)=\sqrt{x^{\prime \prime 2}(s)+y^{\prime \prime 2}(s)+z^{\prime
\prime 2}(s)}\ $\\
$\tau (s)=(\det (\alpha ^{\prime }(s),\alpha
^{\prime \prime }(s),\alpha ^{\prime \prime \prime }(s))/\kappa ^{2}(s)$.
\end{center}

The unit tangent vector, the unit normal vector, the unit binormal vector
and the temporal vector (of the time axis) of the curve are shown by $
E_{1},E_{2},E_{3},E_{4}$ respectively. Thus
\begin{eqnarray*}
E_{1} &=&\alpha ^{\prime }(s)=(x^{\prime }(s),y^{\prime }(s),z^{\prime
}(s),1) \\
E_{2} &=&\frac{E~_{1}^{\prime }}{\kappa \left( s\right) } \\
E_{3} &=&\frac{E\ _{2}^{\prime }}{\tau \left( s\right) } \\
E_{4} &=&\mu \ E_{1}\wedge E_{2}\wedge E_{3}
\end{eqnarray*}
where $\mu $ is chosen as $\mp $1 for det$(E_{1},E_{2},E_{3},E_{4}) $ to be 1.

The Frenet equations in the Galilean 4-space with the spatial Frenet vectors $
E_{1},E_{2},E_{3}$ and the temporal vector $E_{4}$ are given by
\begin{eqnarray}
\mathbf{E}_{1}^{\prime } &=&\kappa \mathbf{E}_{2} \notag \\
\mathbf{E}_{2}^{\prime } &=&-\kappa \mathbf{E}_{1}+\tau \mathbf{E}_{3} \\
\mathbf{E}_{3}^{\prime } &=&-\tau \mathbf{E}_{2} \notag\\
\mathbf{E}_{4}^{\prime } &=&-\sigma \mathbf{E}_{3} \notag
 \label{frenet4}
\end{eqnarray}

Galilean geometry is the study of properties of figures that are invariant
under the Galilean transformations. In general a Galilean transformation in $
n$ spatial dimensions takes the $( n+1) $-vector $( \mathbf{u
},t) $ to the $( n+1) $-vector $( R\mathbf{u}+vt+a, t+a_{0}) $
where $R \in SO(n), v \in R^{n}$and $a \in R^{n}$.
Particularly in $\mathbb{G}^{4}$, a spatial rotation of reference
frame happens for the plane spanned by two spatial axes holding the other
plane stationary.

\subsection{\textbf{Spatial Backlund transformations of curves in the Galilean 4-space:}}

\begin{theorem}
Suppose that $\psi $ is a transformation between two curves $\alpha $ and $
\widetilde{\alpha }$ in the Galilean space $\mathbb{G}^{4}$ with $\widetilde{
\alpha }=\psi (\alpha ).$ We have:
\begin{enumerate}[i.]
\item  The line segment $[\widetilde{\alpha }(s)\alpha (s)]$ at the
intersection of the osculating planes of the curves has constant length $r$
\item The distance vector $\widetilde{\alpha }-\alpha $ has the same Euclidean
angle $\gamma \neq \frac{\pi }{2}$ with the tangent vectors of the curves
\item The binormals of the curves have the same constant Euclidean
angle $\phi \neq 0$.
\end{enumerate}
Then curvatures, torsions and the spatial Backlund transformation of the curves are given by
\begin{eqnarray*}
\tilde{\kappa} &=&-\kappa -2\frac{d\gamma }{ds} \\
\widetilde{\tau } &=&\tau =\frac{\sin \phi }{r} \\
\widetilde{\alpha } &=&\alpha +\frac{2C}{\tau^{2}+C^{2}}(\cos \gamma 
\mathbf{E}_{1}+\sin \gamma \mathbf{E}_{2})
\end{eqnarray*}
where the Backlund parameter is $C=\tau \tan \left( \frac{\phi }{2}\right)$ and $\gamma$ is a solution of the differential equation $\frac{d\gamma }{ds} =\tau \sin \gamma \tan \left( \frac{\phi }{2}\right)-\kappa$.
\end{theorem}
\begin{proof}
Denote by $(\mathbf{E}_{1},\mathbf{E}_{2},\mathbf{E}_{3},\mathbf{E}_{4})$
and $(\widetilde{\mathbf{E}_{1}},\widetilde{\mathbf{E}_{2}},\widetilde{\mathbf{E}
_{3}},\widetilde{\mathbf{E}_{4}})$\ the Frenet frame of the curves $\alpha $ and $
\widetilde{\alpha }$ in the Galilean space $\mathbb{G}^{4}$ respectively. Let $\widetilde{
\mathbf{E}_{3}}$ be the unit binormal of $\widetilde{\alpha }$.

If we denote by $W_{1}$ the unit vector of $\widetilde{\alpha }-\alpha $,
then we can complete $W_{1},\mathbf{E}_{3},\mathbf{E}_{4}$ and $W_{1},
\widetilde{\mathbf{E}_{3}},\widetilde{\mathbf{E}_{4}}$ to the positively
oriented orthonormal frames $(W_{1},W_{2},W_{3},W_{4})$ and $(W_{1},
\widetilde{W_{2}},\widetilde{W_{3}},\widetilde{W_{4}})$ where $W_{3}=\mathbf{
E}_{3},\widetilde{W_{3}}=\widetilde{\mathbf{E}_{3}},W_{4}=\mathbf{E}_{4},
\widetilde{W_{4}}=\widetilde{\mathbf{E}_{4}}$. For a spatial rotation of the $E_{1}E_{2}$ plane holding the other plane constant we can write
\begin{equation*}
\left[ 
\begin{array}{l}
W_{1} \\ 
W_{2} \\ 
W_{3} \\ 
W_{4}
\end{array}
\right] =\left[ 
\begin{array}{cccc}
\cos \gamma & \sin \gamma & 0 & 0 \\ 
-\sin \gamma & \cos \gamma & 0 & 0 \\ 
0 & 0 & 1 & 0 \\ 
0 & 0 & 0 & 1
\end{array}
\right] \left[ 
\begin{array}{l}
\mathbf{E}_{1} \\ 
\mathbf{E}_{2} \\ 
\mathbf{E}_{3} \\ 
\mathbf{E}_{4}
\end{array}
\right]
\end{equation*}
and similarly for $\widetilde{E_{1}}\widetilde{E_{2}}$ plane
\begin{equation*}
\left[ 
\begin{array}{l}
W_{1} \\ 
\widetilde{W_{2}} \\ 
\widetilde{W_{3}} \\ 
\widetilde{W_{4}}
\end{array}
\right] =\left[ 
\begin{array}{cccc}
\cos \gamma & \sin \gamma & 0 & 0 \\ 
-\sin \gamma & \cos \gamma & 0 & 0 \\ 
0 & 0 & 1 & 0 \\ 
0 & 0 & 0 & 1
\end{array}
\right] \left[ 
\begin{array}{l}
\widetilde{\mathbf{E}_{1}} \\ 
\widetilde{\mathbf{E}_{2}} \\ 
\widetilde{\mathbf{E}_{3}} \\ 
\widetilde{\mathbf{E}_{4}}
\end{array}
\right] .
\end{equation*}
Also we can rotate spatially the $\widetilde{W_{2}}\widetilde{W_{3}}$ plane by the transformation 
\begin{equation*}
\left[ 
\begin{array}{l}
W_{1} \\ 
\widetilde{W_{2}} \\ 
\widetilde{W_{3}} \\ 
\widetilde{W_{4}}
\end{array}
\right] =\left[ 
\begin{array}{cccc}
1 & 0 & 0 & 0 \\ 
0 & \cos \phi & -\sin \phi & 0 \\ 
0 & \sin \phi & \cos \phi & 0 \\ 
0 & 0 & 0 & 1
\end{array}
\right] \left[ 
\begin{array}{l}
W_{1} \\ 
W_{2} \\ 
W_{3} \\ 
W_{4}
\end{array}
\right]
\end{equation*}
From the above equations we find the Frenet vectors
\begin{eqnarray}
\widetilde{\mathbf{E}_{1}}&=&(\cos ^{2}\gamma +\sin ^{2}\gamma \cos \phi )
\mathbf{E}_{1}+\cos \gamma \sin \gamma (1-\cos \phi )\mathbf{E}_{2}\notag \\ &&+\sin
\gamma \sin \phi \mathbf{E}_{3}\notag \\ 
\widetilde{\mathbf{E}_{2}}&=&\cos \gamma \sin \gamma (1-\cos \phi )\mathbf{E}
_{1}+(\sin ^{2}\gamma +\cos ^{2}\gamma \cos \phi )\mathbf{E}_{2}\label{frenet4}\\ 
&&-\cos \gamma \sin \phi \mathbf{E}_{3} \notag \\ 
\widetilde{\mathbf{E}_{3}}&=&-\sin \gamma \sin \phi \mathbf{E}_{1}+\sin \phi
\cos \gamma \mathbf{E}_{2}+\cos \phi \mathbf{E}_{3}\notag \\ 
\widetilde{\mathbf{E}_{4}}&=&\mathbf{E}_{4} \notag
\end{eqnarray}
Since $\alpha $ is a unit speed curve, differentiating the below
\begin{equation*}
||\tilde{\alpha}-\alpha ||^{2}=<\tilde{\alpha}-\alpha ,\tilde{\alpha}-\alpha
>=r^{2}
\end{equation*}
and substituting the distance vector
\begin{equation*}
\widetilde{\alpha }-\alpha =r(\cos \gamma \mathbf{E}_{1}+\sin \gamma \mathbf{
E}_{2})
\end{equation*}
we find that $\tilde{\alpha}$ is also a unit speed curve. By the derivative
of above equation and of $\widetilde{\mathbf{E}_{1}}$ with respect to the
arclength s we find
\begin{eqnarray*}
\tau &=&\frac{\sin \phi }{r} \\
\frac{d\gamma }{ds} &=&\tau \sin \gamma \tan \left( \frac{\phi }{2}\right)
-\kappa
\end{eqnarray*}
Also from the derivative of $\widetilde{\mathbf{E}_{3}}$ and use of Frenet
equations gives
\begin{equation*}
\tilde{\tau}=\tau
\end{equation*}
A similar approach for $\widetilde{\mathbf{E}_{1}}$ results in the equality
\begin{equation*}
\tilde{\kappa}=-\kappa -2\frac{d\gamma }{ds}
\end{equation*}
Hence the transformation of the curves becomes
\begin{equation*}
\widetilde{\alpha }=\alpha +\frac{2C}{\tau ^{2}+C^{2}}(\cos \gamma \mathbf{E}
_{1}+\sin \gamma \mathbf{E}_{2})
\end{equation*}
with $C=\tau\tan (\frac{\phi }{2}).$
\end{proof}


\section{\textbf{References}}
\begin{thebibliography}{23}
\bibitem{tenenblat1} K. Tenenblat, C.L. Terng, B\"acklund's theorem for
n-dimensional submanifolds in R$^{2n-1}$, Ann. Math. 111 (1980) 477-490.

\bibitem{chern} S.S.Chern, C.L.Terng, An analogue of B\"acklund's theorem
in affine geometry, Rocky Mt. J. Math. 10 (1980) 105-124.

\bibitem{terng} C.L. Terng, A higher dimension generalization of the
 Sine-Gordon equation and its soliton theory, Ann. Math. 111 (1980) 491-510.

\bibitem{tenenblat2} K. Tenenblat, B\"acklund theorems for submanifolds of
 space forms and a generalized wave equation, Bull. Soc. Brasil. Mat. 16 (1985) 67-92.

\bibitem{palmer} B. Palmer, B\"acklund transformations for surfaces in Minkowski
 space, J. Math. Phys. 31 (1990) 2872-2875.

\bibitem{buyske} S.G. Buyske, Geometric aspects of Backlund transformations
of Weingarten submanifolds, Pacific J. Math. 166 (1994) 213-223.

\bibitem{tian} T. Chou, C. Xifang, B\"acklund transformation on
surfaces with aK+bH=c, Chin. J. Contemp. Math. 18 (1997) 353-364.

\bibitem{chen} W. Chen, H. Li, Weingarten surfaces and
sine-Gordon equation, Sci. China Ser. A, 40 (1997) 1028-1035.

\bibitem{cao} C. Xifang, T. Chou, B\"acklund transformations on surfaces
with (k$_{1}-$ m)(k$_{2}-$ m) = $\pm ~l^{2}$ in $\mathbf{R}^{2,1}$, J.
 Phys. A: Math. Gen, 30 (1997) 6009.

\bibitem{calini} A. Calini, T. Ivey, B\"acklund transformations and
knots of constant torsion, J. Knot Theor. Ramif. 7 (1998) 719-746.

\bibitem{nemeth1}  S.Z. N\'emeth, B\"acklund transformations of n-dimensional
constant torsion curves, Publ. Math-Debrecen 53 (1998) 271-279.

\bibitem{nemeth2} S.Z. N\'emeth, B\"acklund transformations of constant torsion
curves in 3-dimensional constant curvature spaces, Ital. J. Pure Appl. Math. 7 (2000) 125-138.

\bibitem{shief} W.K. Schief, C.Rogers, Binormal motion of curves of
constant curvature and torsion. Generation of soliton surfaces, P. Roy. Soc.
Lond. A Mat. 455 (1999) 3163-3188

\bibitem{chou} T. Chou, Z. Kouhua, T. Yongbo, B\"acklund transformation
on surfaces with constant mean curvature in $\mathbf{R}^{2,1}$,
23B (2003) 369-376

\bibitem{zuo} D. Zuo, Q. Chen, Y. Cheng , B\"acklund theorems in three-dimensional 
de Sitter space and anti-de Sitter space, J. Geom. Phys. 44 (2002) 279-298.

\bibitem{abdel} R.A. Abdel Baky, The Backlund's theorem in Minkowski
3-space  $\mathbf{R}^{3}_{1}$, Appl. Math. Comput. 160 (2005) 41-50.

\bibitem{gurbuz} N. G\"{u}rb\"{u}z, Backlund transformations of constant
torsion curves in $\mathbf{R}_{1}^{n}$, Hadronic J. 29 (2006) 213-220.

\bibitem{ozdemir} M. \"{O}zdemir, A.C. \c{C}\"{o}ken, B\"acklund
transformation for non-lightlike curves in Minkowski 3-space, Chaos
Solitons Fract. 42 (2009) 2540-2545.

\bibitem{tosun} M. Akyi\u{g}it, A.Z. Azak, M. Tosun, Admissible Mannheim curves
in pseudo-Galilean space $\mathbf{G}_{1}^{3}$, http://arxiv.org/abs/1001.2440v3, 2010

\bibitem {yaglom} I.M. Yaglom, A simple non-Euclidean Geometry and its physical basis, Springer-Verlag, New York, 1979

\bibitem{suha} S. Y\.{i}lmaz, Construction of the Frenet-Serret frame of a curve in 4D Galilean space and some applications, Int. J. Phys. Sci. 5 (2010) 1284-1289.

\bibitem {divjak} B. Divjak, The general solution of the Frenet system of differential equations for curves in the pseudo-Galilean space$\mathbf{G}_{1}^{3}$, Math. Commun. 2 (1997) 143-147.

\bibitem{pavkovic} B.J. Pavkovic, I. Kamenarovic, The equiform differential geometry of curves in the Galilean space $\mathbf{G}^{3}$, Glasnik Mat. 22 (1987) 449-457.

\end{thebibliography}
\end{document}